\documentclass{aims1}
\usepackage{amsmath}
  \usepackage{paralist}
  \usepackage{graphics} %% add this and next lines if pictures should be in esp format
  \usepackage{epsfig} %For pictures: screened artwork should be set up with an 85 or 100 line screen
 \usepackage[colorlinks=true]{hyperref}
\hypersetup{urlcolor=blue, citecolor=red}

\def\currentvolume{X}
 \def\currentissue{X}
  \def\currentyear{200X}
   \def\currentmonth{XX}
    \def\ppages{X--XX}

\usepackage{graphicx}
\usepackage[hang,raggedright]{subfigure}

\usepackage{verbatim}
\usepackage{amssymb}
\usepackage{amsfonts,amssymb}
\usepackage{times}
\usepackage{amscd}
\newtheorem{theorem}{Theorem}[section]

\newtheorem{corollary}{Corollary}

\newtheorem{lemma}[theorem]{Lemma}

\theoremstyle{definition}

\newtheorem{definition}[theorem]{Definition}

\newtheorem{remark}{Remark}

\numberwithin{equation}{section}
\newcommand{\ucirc}{\mathbb S}
\newcommand{\uc}{\mathbb S}

\newcommand{\reals}{\mathbb R}
\newcommand{\comps}{\mathbb C}
\newcommand{\C}{\mathcal C}
\newcommand{\ch}{\mbox{$\mathrm{Ch}$}}

\newcommand{\ints}{\mathbb Z}

\newcommand{\A}{\mathcal A}

\newcommand{\AP}{\mathcal A\mathcal P}
\newcommand{\WT}{\mathcal W\mathcal T}

\newcommand{\disk}{\mathbb D}

\newcommand{\si}{\sigma}
\newcommand{\Ta}{\Theta}
\newcommand{\ta}{\theta}
\newcommand{\ga}{\gamma}

\newcommand{\complex}{\mathbb C}
\newcommand{\rsphere}{\complex_\infty}

\newcommand{\e}{\varepsilon}

\newcommand{\bbd}{{\mathbb D}}

\newcommand{\bbc}{\mathbb C}

\newcommand{\W}{\mathcal{W}}

\newcommand{\K}{\mathcal{K}}
\newcommand{\ol}{\overline}

\newcommand{\mk}{\mathbf{k}}

\newcommand{\mc}{\mathbf{c}}
\newcommand{\md}{\mathbf{d}}
\newcommand{\mg}{\mathbf{g}}
\newcommand{\mh}{\mathbf{h}}

\newcommand{\T}{\mathbf{T}}
\newcommand{\G}{{\Gamma}}

\newcommand{\0}{\emptyset}

\newcommand{\bd}{\text{Bd}}

\newcommand{\g}{\mathbf{g}}
\newcommand{\h}{\mathbf{h}}

\newcommand{\Sc}{\mathcal S}
\newcommand{\sm}{\setminus}

\newcommand{\Pro}{\mathcal P}
\newcommand{\al}{\alpha}
\newcommand{\be}{\beta}

\DeclareMathOperator{\diam}{diam}

\title[Laminations and critical portraits]
      {Density of orbits in laminations and the space
of critical portraits}

\author[A. Blokh, C. Curry and L. Oversteegen]{}


 \email{ablokh@math.uab.edu}
 \email{ccurry@huntingdon.edu}
 \email{overstee@math.uab.edu}

\thanks{The first named author is supported by NSF grant DMS-0901038.
The third named author is supported by NSF grant DMS-0906316.}

\begin{document}
\maketitle

% Enter the first author's name and address:
\centerline{\scshape Alexander Blokh}
\medskip
{\footnotesize
% please put the address of the first author
 \centerline{Department of Mathematics}
   \centerline{University of Alabama at Birmingham}
   \centerline{ Birmingham, AL 35294-1170, USA}
} % Do not forget to end the {\footnotesize by the sign }

%\author{Alexander~Blokh} \thanks{The authors were partially supported
 % by NSF grant DMS-0901038 (A.B.)  and by NSF grant DMS-0906316
 % (L.O.)}

 \medskip

\centerline{\scshape Clinton Curry}
\medskip
{\footnotesize
 % please put the address of the second  and third author
 \centerline{Department of Mathematics}
   \centerline{1500 East Fairview Avenue}
   \centerline{Montgomery, AL 36106, USA}
}

\medskip

% Enter the first author's name and address:
\centerline{\scshape Lex Oversteegen}
\medskip
{\footnotesize
% please put the address of the first author
 \centerline{Department of Mathematics}
   \centerline{University of Alabama at Birmingham}
   \centerline{ Birmingham, AL 35294-1170, USA}
}

\bigskip

%\author{Clinton~Curry}

%\author{Lex~Oversteegen}

%\address[Alexander~Blokh, Lex~Oversteegen and Clinton~Curry]
%{Department of Mathematics\\ University of Alabama at Birmingham\\
 % Birmingham, AL 35294-1170\\ Tel.: (205)934-2154\\ Fax:
 % (205)934-9025}

%\email[Alexander~Blokh]{ablokh@math.uab.edu}
%\email[Lex~Oversteegen]{overstee@math.uab.edu}
%\email[Clinton~Curry]{ccurry@huntingdon.edu}

% The name of the associate editor will be entered by an editorial staff
% "Communicated by the associate editor name" is not needed for special issue.
 \centerline{(Communicated by the associate editor Sebastian van Strien)}

%The abstract of your paper
\begin{abstract}
Thurston introduced $\si_d$-invariant laminations (where $\si_d(z)$
  coincides with $z^d:\ucirc\to \ucirc$, $d\ge 2$). He defined
  \emph{wandering $k$-gons} as sets $\T\subset \ucirc$ such that
  $\si_d^n(\T)$ consists of $k\ge 3$ distinct points for all $n\ge 0$
  and the convex hulls of all the sets $\si_d^n(\T)$ in the plane are
  pairwise disjoint. Thurston proved that $\si_2$ has no wandering $k$-gons
  and posed the problem of their existence for $\si_d$,\, $d\ge 3$.

  Call a lamination with wandering $k$-gons a \emph{WT-lamination}.
  Denote the set of cubic critical portraits by $\A_3$. A critical
  portrait, compatible with a WT-lamination, is called a
  \emph{WT-critical portrait}; let $\WT_3$ be the set of all of them.
  It was recently shown by the authors that cubic WT-laminations exist
  and cubic WT-critical portraits, defining polynomials with
  \emph{condense} orbits of vertices of order three in their dendritic
  Julia sets, are dense and locally uncountable in $\A_3$ ($D\subset
  X$ is \emph{condense in $X$} if $D$ intersects every subcontinuum of
  $X$).  Here we show that $\WT_3$ is a dense first category subset of
  $\A_3$, that critical portraits, whose laminations
  have a condense orbit in the topological Julia set, form a residual
  subset of $\A_3$, and that the existence of a condense orbit in the Julia
  set $J$ implies that $J$ is locally connected.
\end{abstract}

%\date{April 20, 2011; revised version November 22, 2011}

%\keywords{Complex dynamics, locally connected, Julia set, lamination,
%  wandering gaps}

%\subjclass[2000]{Primary 37F20;
%Secondary 37B45, 37F10, 37F50}

\section{Introduction}\label{intro}

Let $\complex$ be the complex plane and
$\rsphere=\complex\cup\{\infty\}$ be the complex sphere.  The
following result is a special case of a theorem due to Thurston
\cite{thu85}.

\begin{theorem}[No Wandering Vertices for Quadratics]\label{t:nowanpol}
  Let $P(z) = z^2 + c$ be a polynomial with connected Julia set $J_P$.
  If $z_0 \in J_P$ is a point such that $J_P \setminus \{z_0\}$ has at
  least three components, then $z_0$ is either preperiodic or
  eventually maps to the critical point $0$.
\end{theorem}

In \cite{bo04} we construct an uncountable family of cubic polynomials
$P$ with $z_0\in J_P$ such that $J_P \setminus \{z_0\}$ has three
components and $z_0$ is neither preperiodic nor precritical; such a
point is called a \emph{wandering vertex}. In \cite{bco10}, we improve
on these results by finding a collection of polynomials, dense in the
appropriate parameter space, with wandering vertices whose orbits have
a property that we call \emph{condensity}.

\begin{definition}\label{condens}
  For a topological space $X$ a set $A\subset X$ is
  \emph{continuumwise dense} (abbreviated \emph{condense}) in $X$ if
  $A\cap Z\ne \0$ for each non-degenerate continuum $Z\subset X$.  A
  map $f:X \to X$ is also called condense if there exists $x_0 \in X$
  such that $\{f^n(x_0)\,\mid\,n \ge 0\}$ is condense in $X$.
\end{definition}

It is not hard to see that condensity is much stronger than density.
% For example, most subcontinua of dendrites are nowhere dense; though
% the set of endpoints of a dendrite is residual if it is dense, it is
% never condense.
For example, if $J$ is a Julia set from the real quadratic family
which is not homeomorphic to an interval, the set of endpoints is
dense in $J$, but not condense. Moreover, in this case the set of
transitive points (i.e., points with dense orbit in $J$) is a subset
of the endpoints of $J$, so such maps are not condense.

To state the results of \cite{bco11} precisely, we must indicate in
which parameter space we are working. Polynomials are
naturally %parameterized by
associated to \emph{critical portraits}, introduced by Yuval Fisher in
his Ph.D. thesis \cite{fish89}.  Let $\sigma_d:\ucirc \to \ucirc$ be
the angle $d$-tupling map $\si_d(z)=z^d$.  A degree $d$ critical
portrait, loosely speaking, is a \emph{maximal} collection $\Theta =
\{\Theta_1, \ldots, \Theta_n\}$ of sets of angles in $\ucirc$ which
are pairwise disjoint, pairwise unlinked (i.e., have disjoint convex
hulls in $\overline{\mathbb D}$ when angles are interpreted as points
in $\ucirc$), and such that $\sigma_d(\Theta_i)$ is a point for each
$\Theta_i \in \Theta$ (it is easy to see that $\sum (|\Ta_i|-1)=d-1$).
% Thus, this models the situation for the critical points of dendritic
% Julia sets: if $c_1, \ldots, c_n$ are the critical points, and
% $\Theta_i$ is the set of angles whose external rays land at $c_i$,
% then $\Theta = \{\Theta_1, \ldots, \Theta_n\}$ is a critical
% portrait.)

This notion is used to capture the location of critical points. The
set of all critical portraits of degree $d$ is denoted $\A_d$, and is
naturally endowed with a topology (see Definition~\ref{cu} for
details). We say that a \emph{critical portrait $\Ta$ corresponds to a
  polynomial $P$} with dendritic Julia set if for each $\Ta_i \in \Ta$
there is a distinct critical point $c_i \in J_P$ such that the
external rays whose angles are in $\Ta_i$ land at $c_i$ (see
Section~\ref{sec:critport} for more information). Now we state the
main theorem of \cite{bco10}.

\begin{theorem}[\cite{bco10}]\label{main1}
  $\A_3$ contains a dense locally uncountable set
  $\{\Ta_\alpha\,\mid\, \al\in A\}$ of critical portraits such that
  for each $\al\in A$ the following holds:
  \begin{itemize}
  \item $\Ta_\alpha$ corresponds to a polynomial $P_\alpha$ with
    dendritic Julia set $J_{P_\alpha}$,
  \item $\{P_\alpha|_{J_{P_\alpha}}\}$ are pairwise non-conjugate, and
  \item $J_{P_\alpha}$ contains a wandering vertex with condense
    orbit.
  \end{itemize}
\end{theorem}

The aim of this paper is to further investigate the notions and
objects studied in Theorem~\ref{main1}, such as condensity and the set
of critical portraits which correspond to polynomials with wandering
vertices. To explain our results, we recall constructions from
\cite{kiwi97, bco11}: given a polynomial $P$ with connected Julia set
$J_P$, one can construct a corresponding locally connected continuum
$J\subset \comps$ (called a \emph{topological Julia set}) and branched
covering map $f:\comps\to \comps$ (called a \emph{topological
  polynomial}) so that $P$ is \emph{monotonically semiconjugate} to
$f$ (i.e., there exists a monotone map $m:\comps\to\comps$ such that
$m\circ P=f\circ m$) and $J=m(J_P)$. We refer to $f|_J$ as the
\emph{locally connected model} of $P$. It is known \cite{bo10} that in
some cases $J$ is a single point.

Let us describe the organization of the paper and the main results.
% Suppose $\Theta\in\A_d$ is a critical portrait consisting of $d-1$
% critical leaves with dense orbits. Then the lamination $\sim_\Theta$
% generated by $\Theta$ is condense. Moreover, there exists a
% polynomial $P$ with $\sim_P=\sim_\Theta$ and $J_P$ locally connected
% so that $P$ and $f_{\sim_\Theta}$ are conjugate and $P|_{J_P}$ is
% condense. All such critical portraits $\Ta$ form a residual subset
% of $\A_d$.
After discussing %some
preliminary notions and history in Section~\ref{results}, we study
properties of condense maps in Section~\ref{s:condens}. In particular
we show in Theorem~\ref{t:condens1} that polynomials which admit
condense orbits either in their Julia sets (or in some circumstances
their locally connected models) have locally connected Julia sets. In
Section~\ref{famwt} we prove that the set of cubic critical portraits
corresponding to polynomials with condense orbits in their Julia sets
is residual in $\A_3$ (Theorem~\ref{t:condenscript}), while the set of
critical portraits which correspond to polynomials with wandering
vertices is meager (Theorem~\ref{t:meager}).

\section{Preliminaries}\label{results}

\subsection{Laminations}\label{sec:term}
In what follows, we parameterize the circle as $\ucirc = \reals /
\ints$, so the total arclength of $\ucirc$ is $1$. The \emph{positive}
direction on $\ucirc$ is the counterclockwise direction, and by the
arc $(p,q)$ in the circle we mean the positively oriented arc from $p$
to $q$. A \emph{(strictly) monotone} map $g:(p, q)\to \uc$ is a map
(strictly) monotone at each point of $(p, q)$ in the sense of positive
direction on $\uc$. By $\ch(A)$ we denote the \emph{convex hull} of a
set $A\subset \complex$ and by $|B|$ we denote the cardinality of the
set $B$.

Laminations are combinatorial structures on the unit circle,
introduced by Thurston \cite{thu85} as a tool for studying individual
complex polynomials $P:\rsphere \to \rsphere$ and the space of all of
them.  In what follows, we use a restricted formulation of this
concept.  Specifically, \cite{thu85} defines a lamination as a
collection of chords in the unit disk, satisfying certain dynamical
properties.  These collections are intended to reflect the pattern of
external rays landing in the Julia set.  We instead use the
formulation contained in \cite{kiwi97}, explained in detail below,
based on pinched disk models (see \cite{dou93} for more information
in this vein, and also ~\cite{blolev99}).

Let $P$ be a degree $d$ polynomial with a locally connected
(and hence connected) Julia set $J_P$; we will recall how to associate
an equivalence relation $\sim_P$ on $\ucirc$ to $P$, called the
\emph{$d$-invariant lamination generated by $P$}.
The filled-in Julia set $K_P$ is compact, connected, and full, so its
complement $\rsphere \setminus K_P$ is conformally isomorphic to the
open unit disk $\bbd$.  By \cite[Theorem 9.5]{milnor}, there is a
particular conformal isomorphism $\Psi:\bbd \to \rsphere \setminus
K_P$ so that $\Psi$ conjugates $\si_d(z)=z^d$ on $\bbd$ to
$P|_{\rsphere \setminus K_P}$ (i.e., $\Psi(z^d) = (P|_{\rsphere
  \setminus K_P} \circ \Psi)(z) \text{ for } z \in \bbd$). When $J_P$
is locally connected, $\Psi$ extends to a continuous map $\ol{\Psi}:
\ol{\bbd}\to \ol{\rsphere\setminus K_P}$ which semiconjugates $z
\mapsto z^d$ on $\ol{\bbd}$ to $P|_{\ol{\rsphere \setminus K_P}}$. Let
$\psi:\ucirc \to J_P$ denote the restriction $\ol \Psi
|_{\ucirc}$. Define the equivalence $\sim_P$ on $\ucirc$ so that $x
\sim_P y$ if and only if $\psi(x)=\psi(y)$; this equivalence relation
is the aforementioned $d$-invariant lamination generated by $P$. The
quotient space $\ucirc/\sim_P=J_{\sim_P}$ is homeomorphic to $J_P$ and
the \emph{induced map} $f_{\sim_P}:J_{\sim_P}\to J_{\sim_P}$ defined
by $f_{\sim_P} = \psi \circ \si_d \circ \psi^{-1}$ is conjugate to
$P|_{J_P}$.

Kiwi \cite{kiwi97} extended this construction to polynomials $P$ with
no irrationally neutral cycles and introduced a similar $d$-invariant
lamination $\sim_P$. Then $J_{\sim_{P}}=\ucirc/\sim_P$ is locally
connected and $P|_{J_{P}}$ is semi-conjugate to $f_{\sim_P}$ by a
\emph{monotone} map $m:J_p\to J_{\sim_P}$, i.e., a map $m$ whose point
preimages are connected. This was extended in \cite{bco11} to all
polynomials $P$ with connected $J_P$. The lamination $\sim_P$
combinatorially describes the dynamics of $P|_{J_{P}}$.

One can introduce abstract laminations (frequently denoted by $\sim$)
as equivalence relations on $\ucirc$ having properties in common with
laminations generated by polynomials as above. Consider an equivalence
relation $\sim$ on the unit circle $\ucirc$.  Equivalence classes of
$\sim$ will be called \emph{($\sim$-)classes} and will be denoted by
boldface letters. A $\sim$-class consisting of two points is called a
\emph{leaf}; a class consisting of at least three points is called a
\emph{gap} (this is more restrictive than Thurston's definition in
\cite{thu85}). Fix an integer $d>1$.  Then $\sim$ is said to be a
\emph{$d$-invariant lamination} if:

\begin{itemize}
\item[(E1)] $\sim$ is \emph{closed}: the graph of $\sim$ is a closed
  set in $\ucirc \times \ucirc$;
\item[(E2)] $\sim$-classes are \emph{pairwise unlinked}: if $\g_1$ and
  $\g_2$ are distinct $\sim$-classes, then their convex hulls
  $\ch(\g_1), \ch(\g_2)$ in the unit disk $\bbd$ are disjoint;
\item[(E3)] $\sim$-classes are either \emph{totally disconnected} (and
  hence $\sim$ has uncountably many classes) or the entire circle
  $\ucirc$ is one class;
\item[(D1)] $\sim$ is \emph{forward invariant}: for a class $\g$, the
  set $\si_d(\g)$ is also a class;
\item[(D2)] $\sim$ is \emph{backward invariant}: for a class $\g$, its
  preimage $\si_d^{-1}(\g)=\{x\in \ucirc: \si_d(x)\in \g\}$ is a union
  of classes; and
\item[(D3)] for any gap $\g$, the map $\si_d|_{\g}: \g\to \si_d(\g)$
  is a \emph{covering map with positive orientation}, i.e., for every
  connected component $(s, t)$ of $\ucirc\setminus \g$ the arc
  $(\si_d(s), \si_d(t))$ is a connected component of $\ucirc\setminus
  \si_d(\g)$.
\end{itemize}
Notice that (D2) and (E3) follow from (D1).

Call a class $\g$ \emph{critical} if $\si_d|_{\g}: \g\to \si_d(\g)$ is
not one-to-one, and \emph{precritical} if $\si_d^j(\g)$ is critical
for some $j\ge 0$. Call $\g$ \emph{preperiodic} if
$\si^i_d(\g)=\si^j_d(\g)$ for some $0\le i<j$. A gap $\g$ is
\emph{wandering} if $\g$ is neither preperiodic nor precritical. Let
$J_\sim = \ucirc / \sim$, and let $\pi_\sim: \ucirc \to J_\sim$ be the
corresponding quotient map.  The map $f_\sim:J_\sim \to J_\sim$
defined by $f_\sim = \pi_\sim \circ \si_d \circ \pi_\sim^{-1}$ is the
map \emph{induced} on $J_\sim$ by $\sigma_d$. Then we call $f_\sim$ a
\emph{topological polynomial}, and $J_\sim$ a \emph{topological Julia
  set}.

\subsection{Bounds for wandering classes}

J. Kiwi~\cite{kiwigap} extended the No Wandering Triangles Theorem by
showing that a wandering gap in a $d$-invariant lamination is at most
a $d$-gon. Thus all infinite $\sim$-classes (and Jordan curves in
$J_\sim$) are preperiodic. In \cite{lev1} G. Levin showed that
laminations with one critical class have no wandering gaps. For a
lamination $\sim$, let $k_\sim$ be the size of a maximal collection of
non-degenerate $\sim$-classes whose $\si_d$-images are points and whose
orbits are infinite and pairwise disjoint. Also, let $N_\sim$ be the number of
cycles of infinite $\sim$-classes plus the number of cycles of Jordan
curves in $J_\sim$.

\begin{theorem}[\cite{blolev99}]\label{bl}
  Let $\sim$ be a $d$-invariant lamination and let $\G$ be a
  \emph{non-empty} collection of wandering $d_j$-gons ($j=1,2,\dots$)
  with distinct grand orbits. Then $\sum_{j} (d_j - 2)\le k_\sim-1$
  and $\sum_{j} (d_j - 2)+N_\sim\le d-2$. In particular, in the cubic
  case if $\G$ is non-empty, then it must consist of one
  non-precritical $\sim$-class with three elements, all $\sim$-classes
  are finite, $J_\sim$ is a dendrite, and both critical classes are
  leaves with disjoint forward orbits.
\end{theorem}

\subsection{Critical portraits}\label{sec:critport}

Following \cite{thu85} and \cite{hubbdoua85a, hubbdoua85b} we look at the set $\C_d$
from infinity and consider the \emph{shift locus}, which is the set
$\Sc_d$ of polynomials whose critical points escape to infinity. The
set $\Sc_d$ is the unique hyperbolic component of $\Pro_d$ consisting
of polynomials with all cycles repelling. It is not known if all
polynomials with all cycles repelling belong to the set $\ol
{\Sc_d}$. Looking at $\C_d$ from infinity means studying locations of
polynomials in $\Sc_d$ depending on their dynamics and using this to
describe the polynomials belonging to $\ol {\Sc_d}\cap \C_d$. A key
tool in studying $\C_d$ is \emph{critical portraits}, introduced in
\cite{fish89}, and widely used afterward (see, e.g., \cite{bfh92,
  po93, golmil93} and \cite{ki4}). We now recall some standard
material; % (compare \cite{ki4, bfh92, golmil93, po09});
here we closely follow \cite[Section 3]{ki4}. Call a chord with
endpoints $a, b\in \ucirc$ \emph{critical} if $\si_d(a)=\si_d(b)$.

% A key tool here is \emph{critical portraits}; to describe them, we
% give a brief overview of \cite{ki4}. Call a chord with endpoints $a,
% b\in \ucirc$, \emph{critical} if $\si_d(a)=\si_d(b)$.

\begin{definition}\label{cp}
  A \textbf{critical portrait} is a collection $\Theta=\{\Theta_1,
  \dots, \Theta_n\}$ of finite subsets of $\ucirc$ such that the
  following hold:
  \begin{enumerate}
  \item the boundary of the convex hull $\ch(\Theta_i)$ of every set
    $\Theta_i$ consists of critical chords (under $\si_d$);
  \item the sets $\Theta_1, \dots, \Theta_n$ are \textbf{pairwise
      unlinked} (that is, convex hulls of the sets $\Theta_i$ are
    pairwise disjoint); and
  \item $\sum (|\Ta_i|-1)=d-1$.
  \end{enumerate}
\end{definition}

The sets $\Ta_1, \dots, \Ta_n$ are called the \emph{initial sets of
  $\Ta$} (or \emph{$\Ta$-initial sets}). Denote by $A(\Ta)$ the union
of all angles from the initial sets of $\Ta$. The convex hulls of the
$\Ta$-initial sets divide the rest of the open unit disk into
components. In Definition~\ref{ue}, points of $\ucirc\setminus A(\Ta)$
are declared equivalent if they belong to the boundary of one such
component; we do not assume that $\Ta$ is a critical portrait because
we need this equivalence later in a more general situation.

\begin{definition}\label{ue}
  Let $\Ta$ be a finite collection of pairwise unlinked finite subsets
  of $\ucirc$. Angles $\alpha, \beta\in \ucirc \setminus A(\Ta)$ are
  \textbf{$\Ta$-unlinked equivalent} if $\{\alpha, \beta\},$ $\Ta_1,$
  $\dots,$ $\Ta_n$ are pairwise unlinked. The equivalence classes
  $L_1(\Ta),$ $\dots,$ $L_d(\Ta)$ are called \textbf{$\Ta$-unlinked
    classes}. Each $\Ta$-unlinked class $L$ is the intersection of
  $\uc\sm A(\Ta)$ with the boundary of a component of $\disk\sm
  \bigcup \ch(\Ta_i)$. In the degree $d$ case, each $\Ta$-unlinked
  class of a critical portrait $\Ta$ is the union of finitely many
  \emph{open} arcs of total length $1/d$. Thus, there are $d$
  $\Ta$-unlinked classes.
\end{definition}

\begin{definition}[compact-unlinked topology \cite{ki4}]\label{cu}
  Define the space $\A_d$ as the set of all critical portraits endowed
  with the \textbf{compact-unlinked} topology generated by the
  subbasis $V_X=\{\Ta\in \A_d: X\subset L_\Ta\}$ where $X\subset \uc$
  is closed and $L_\Ta$ is a $\Ta$-unlinked class.
\end{definition}

Note for example that $\A_2$ is the quotient of $\uc$ with antipodal
points identified, so it is homeomorphic to the unit circle. For a
critical portrait $\Ta$, a lamination $\sim$ is called
\emph{$\Ta$-compatible} if all $\Ta$-initial sets are contained in
$\sim$-classes; if there is a $\Ta$-compatible WT-lamination, $\Ta$ is
said to be a \emph{WT-critical portrait}. The trivial lamination,
identifying all points of $\uc$, is compatible with any critical
portrait.

To define \emph{critical portraits with aperiodic kneading}, let us
introduce the notion of a \emph{one-sided itinerary} for $t\in \ucirc$
(see \cite{ki4}). Given a critical portrait $\Ta=\{\Ta_1, \dots,
\Ta_d\}$ with $\Ta$-unlinked classes $L_1(\Ta), \dots, L_d(\Ta)$ and
$\ta\in\ucirc$, define
% \[i^{\pm}(\ta)=(i_0,i_1,\dots), \text{ with } i_j\in\{1,\dots, d\}\]
$i^{+}(\ta)$ (respectively, $i^-(\ta)$) as the sequence
$(i_0,i_1,\dots)$, with $i_j\in\{1,\dots, d\}$ such that there are
$y_n\searrow \ta$ (respectively, $y_n \nearrow \ta$) with
$\si_d^j(y_n)\in L_{i_j}(\Ta)$ for $n$ sufficiently large. Also,
define the itinerary $i(\ta)$ as a sequence $I_0I_1\dots$ such that
each $I_j$ is the set from $\Ta \cup \{L_1(\Ta),$ $\dots,$
$L_d(\Ta)\}$ to which $\si^j(\ta)$ belongs. An angle $\ta\in \ucirc$
is said to have a \emph{periodic kneading} if $i^+(\ta)$ or $i^-(\ta)$
is periodic. A critical portrait $\Ta$ is said to have \emph{aperiodic
  kneading} if no angle from $A(\Ta)$ has periodic kneading. The
family of all degree $d$ critical portraits with aperiodic kneading is
denoted by $\AP_d$.

\begin{definition}[\cite{kiwi97, ki4}]\label{lamtheta}
  The lamination $\sim_\Ta$ is the smallest closed equivalence
  relation identifying any pair of points $x, y \in \ucirc$ where
  $i^+(x)=i^-(y)$.  By Kiwi \cite{kiwi97, ki4}, for any critical
  portrait $\Ta$ the relation $\sim_\Ta$ is a $\Ta$-compatible
  lamination; it is said to be \emph{generated} by $\Ta$.
\end{definition}

Critical portraits reflect the landing patterns of the external rays
at the critical points.  By Kiwi \cite{ki4}, a nice correspondence
between critical portraits of degree $d$ and the set $\ol {\Sc_d}\cap
\C_d$ associates to each critical portrait $\Ta\in \A_d$ a connected
set $I(\Ta)\subset \ol {\Sc_d} \cap \C_d$, called the \emph{impression
  of $\Ta$}, such that the dynamics of a polynomial in $I(\Ta)$ is
closely related to the properties of $\Ta$. The relation is especially
nice when $\Ta$ has aperiodic kneading. The following fundamental
result of Kiwi \cite{kiwi97, ki4} explicitly lists properties of
critical portraits with aperiodic kneading and their connections to
polynomials.

\begin{theorem}\label{kiwistruct}
  Let $\Ta\in \AP_d$.  Then $\sim_\Ta$ is the unique
  $\Ta$-compatible invariant lamination. The quotient $J_{\sim_\Ta}$
  is a non-degenerate dendrite, and all $\sim$-classes are finite.
  Furthermore, there exists a polynomial $P$ whose Julia set $J_P$
  is a non-separating continuum in the plane and $P|_{J_P}$ is
  monotonically semiconjugate to $f_{\sim_\Ta}|_{J_{\sim_\Ta}}$. The
  semiconjugating map $m_{\Ta, P}=m:J_P \to J_{\sim_\Ta}$ maps
  impressions of external angles to points and maps the set of
  $P$-preperiodic points in $J_P$ bijectively to the set of
  $f_{\sim_\Ta}$-preperiodic points. Moreover, $J_P$ is locally
  connected at all $P$-preperiodic points.
\end{theorem}

In the situation of Theorem~\ref{kiwistruct} polynomials $P$ such that
$P|_{J_P}$ is monotonically semiconjugate to
$f_{\sim_\Ta}|_{J_{\sim_\Ta}}$ are said to be \emph{associated to the
  critical portrait $\Ta$}.

\subsection{Monotone models for connected Julia sets}\label{s:mm}

As was explained in Section~\ref{intro}, the main results of
\cite{kiwi97, bco11} yield a locally connected model for the
restriction of a polynomial to its connected Julia set. We will need a
detailed version of these results stated below in
Theorem~\ref{t:model}.

\begin{theorem}[\cite{kiwi97, bco11}]\label{t:model}
  Let $P$ be a degree $d$ polynomial with connected Julia set
  $J_P$. Then there exists a $d$-invariant lamination $\sim$
  and a monotone onto map $M_P:\comps \to \comps$ with
  the following properties.
  \begin{enumerate}
  \item $J_\sim = M_P(J_P)$ and $J_P\subset M_P^{-1}(J_\sim) \subset K_P$.
  \item $M_P$ sends impressions of $J_P$ to points.
  \item $m_P=M_P|_{J_P}$ is the finest monotone map of $J_P$ onto a
    locally connected continuum (i.e., if $\psi:J_P \to T$ is a
    monotone map onto a locally connected continuum $T$, then there is
    a monotone map $\psi':J_\sim \to T$ such that $\psi = \psi' \circ
    m_P$).
  \item $M_P$ semiconjugates $P$ to a branched covering map $g_P:\comps
    \to \comps$ under which $J_\sim$ is fully invariant so that
$g_P|_{J_\sim}$ is conjugate to the topological polynomial $f_\sim$.
  \end{enumerate}
\end{theorem}

\begin{remark}\label{rem:kiwimodel}
Suppose that $\Ta \in \AP_d$ is associated to the polynomial $P$; let
us show that the lamination $\sim_\Ta$ defined in
Theorem~\ref{kiwistruct} and the lamination $\sim_P$ defined in
Theorem~\ref{t:model} coincide.  Indeed, by Theorem~\ref{t:model}
there exists a monotone map $\psi':J_{\sim_P}\to J_{\sim_\Ta}$. If
this map is not a homeomorphism, it will collapse a non-degenerate
subcontinuum $Q\subset J_{\sim_P}$ to a point $x\in J_{\sim_\Ta}$.
Since impressions map to points of $J_{\sim_P}$, infinitely many
distinct impressions of external rays are contained in the fiber
$m_{\Ta, P}^{-1}(x)$ which by Theorem~\ref{kiwistruct} implies that
the $\sim_\Ta$-class corresponding to $x$ is infinite.  This
contradicts Theorem~\ref{kiwistruct}, which states that
$\sim_\Ta$-classes are finite.
\end{remark}

Theorem~\ref{t:model} establishes the semiconjugacy $m_P$ on the
\emph{entire complex plane}, so that $m_P$-images of external rays to
$J_P$ are curves in $\comps$ accumulating on points of
$J_{\sim_P}$. For $x\in J_{\sim_P}$, the set $m_P^{-1}(x)\cap J_P$ is
the union of impressions of angles $\al$ such that $m_P(R_\al)$ lands
on $x$ . The \emph{order} of $x$ in $J_{\sim_P}$ is the number of
components of $J_{\sim_P}\sm \{x\}$ and can be either a finite number
or infinity. By Theorem~\ref{t:model} if the order of $x$ in
$J_{\sim_P}$ is finite then it equals the number of angles with
impressions in $m_P^{-1}(x)$ (or equivalently the number of angles
whose impressions intersect $m_P^{-1}(x)$). If the order of $x$ in
$J_{\sim_P}$ is infinite, then there are infinitely many angles with
impressions in $m_P^{-1}(x)$.

\section{Condensity}\label{s:condens}

We begin with a few lemmas concerning the dynamics of a condense
topological polynomial.  If $J$ is a dendrite, by $[a, b]_J$ we mean
the unique arc in $J$ connecting the points $a, b\in J$. A continuum
$X\subset \complex$ is called \emph{unshielded} if it coincides with
the boundary of the unique unbounded component of $\complex\sm X$.
Note that all connected Julia sets of polynomials and all
topological Julia sets are unshielded continua. A point $x\in X$ is
called a \emph{cutpoint} of $X$ if $X\sm \{x\}$ is not connected. In
what follows a lamination $\sim$ such that $f_\sim$ is condense is
called \emph{condense}; also, a critical portrait compatible with a
condense lamination is said to be \emph{condense}.

\begin{lemma}\label{l:densco}
If $X\subset \complex$ is an unshielded locally connected continuum and
$A\subset X$ is connected and dense in $X$, then $A$ is condense in $X$
and contains all cutpoints of $X$.
\end{lemma}

\begin{proof}
%Note first that the result is true if $X$ is a simple
%closed curve or an arc.  Hence we may assume below that $X$ contains
%a cutpoint of order at least three.
% I have substituted an equivalent proof which does not refer to the
% "laminational" picture.
If $Z\subset X$ is a closed set with $X\sm Z$ disconnected, then all
components of $X\sm Z$ are open. Hence all such components intersect
$A$. Since $A$ is connected, this implies that $A\cap Z\ne \0$.
Suppose that $A$ is not condense in $X$. Then there exists an arc
$I\subset X$ disjoint from $A$. Note that $X\setminus I$ is open and
connected (by virtue of containing $A$). Therefore $X\setminus I$ is
path connected and locally path connected. %Hence, there is an arc
%$J\subset X$ joining the endpoints of $I$ in $X\setminus I$.
It follows that there exists a simple closed curve $S\subset X$
which contains a non-degenerate subsegment $I'$ with endpoint $a',
b'$ of $I$. The curve $S$ encloses a topological disk $U$. Clearly,
any two-point set $\{a, b\}\subset S$ separates $X$ (two external
rays landing at $a$ and $b$ and an arc inside $U$ from $a$ to $b$
disconnect $\complex$). Hence $A\cap \{a', b'\}\ne \0$, a
contradiction.

\end{proof}

Let us now study condensity in the context of laminations. We call a
lamination $\sim$ \emph{degenerate} if the whole $\uc$ forms a
$\sim$-class (and so $J_\sim$ is a point); we call $\sim$
\emph{trivial} if all $\sim$-classes are singletons (and
$J_\sim=\uc$).

\begin{lemma}\label{l:condens0}
  Let $\sim$ be a condense lamination. Then either $\sim$ is
  degenerate, or $\sim$ is trivial, or $J_\sim$ is a dendrite.
\end{lemma}

\begin{proof}
  Suppose $J_\sim$ is non-degenerate and let $x \in J_\sim$ be a point
  with condense orbit. If $J_\sim$ is not a dendrite, then it contains
  a Jordan curve. By \cite{blolev99} it follows that $J_\sim$ contains
  a periodic Jordan curve $B$ of period, say, $k$.  Since $x$ must
  enter $B$, it follows that the union of $\bigcup^k_{i=1}
  f_\sim^i(B)=J_\sim$. Since $J_\sim$ is a topological Julia set, it
  is easy to see that then $J_\sim$ is the unit circle and the
  lamination $\sim$ is trivial.
\end{proof}

%As in the case when $\sim$ is degenerate the claims we are interested
%in are trivial, we will assume from now on that $J_\sim$ is a dendrite.

%% We can't do this, because later we make the distinction!
% To lighten the notation, from now on we use notation $J=J_\sim$, $f=f_\sim$.

\begin{lemma}\label{l:condens1}
  Suppose that $K\subset J_\sim$ is a continuum with dense orbit and
  that $f^n(K) \cap K \neq \emptyset$. If $t \ge 0$ is an integer, the
  union $\bigcup^\infty_{j=0} f_\sim^{nj+t}(K)$ is a condense
  connected subset of $J$ containing all cutpoints of
  $J_\sim$. Further, if $f^n(K) \subset K$, then $K=J_\sim$.
\end{lemma}

Observe that in this lemma we do not assume that $f$ is condense.

\begin{proof}
  Under the hypotheses, $A_0=\bigcup f_\sim^{nk}(K)$ is a connected
  subset of $J_\sim$, and so are the sets $A_l=\bigcup
  f_\sim^{nk+l}(K)$ where $1 \le l \le n-1$.  By the assumption, the
  union $A=\cup^{n-1}_{l=0}A_l$ is dense in $J$. Observe that
  $f_\sim(A_l)\subset A_{l+1}$, where indices are interpreted modulo
  $n$.

  Since $\bigcup_{l=0}^{n-1} \overline{A_l} = J_\sim$ it follows from
  the Baire Category Theorem that some $\ol{A_s}$ contains an open
  subset of $J_\sim$.  Since $f_\sim$ eventually maps any open set
  onto $J_\sim$, it follows that $f_\sim^r(\ol{A_s})=J_\sim$ for
  some $r\ge 0$. Hence, for all $i\ge 0$,
  $f^{r+i}_\sim(\ol{A_s})=\ol{A_{s+i}}=J_\sim$, and so for any $t$ the
  set $A_t$ is connected and dense in $J_\sim$.  Then
  Lemma~\ref{l:densco} implies that $A_t$ is condense and contains all
  cutpoints of $J_\sim$.

  In the case that $f_\sim^n(K) \subset K$, it follows that $A_0
  \subset K$; that $K$ is closed and $A_0$ is dense implies that $K
  = J_\sim$.
  \begin{comment}
    % Old proof
    Note that $f_\sim$ is an open map, so $A_0$ is nowhere dense if
    and only if $A_l$ is nowhere dense for each $l < n$.  If $A_0$ is
    nowhere dense in $J_\sim$, then by the above $A_{n-1}$ is nowhere
    dense either, then by induction all sets $A_l$ are nowhere dense,
    and so is their union $A$, a contradiction. Hence $\ol{A_0}$ has a
    non-empty interior. Since $f_\sim^l$ eventually maps any open set
    onto $J_\sim$, it follows that $\ol{A_0}=J_\sim$, and similarly,
    $\ol{A_l}=J_\sim$.  Lemma~\ref{l:densco} now implies the first
    claim. To prove the second claim apply the first one to $K$. Then
    the set $A_0$ constructed above coincides with $K$; thus, $K$ is a
    condense subset of $J_\sim$ which implies that $K=J_\sim$.
  \end{comment}
\end{proof}

The next lemma shows that condense maps resemble transitive maps.
Recall that any topological polynomial on a dendrite must have fixed
cutpoints (see, e.g., \cite{thu85, bfmot11}).

\begin{lemma}\label{l:condens2}
  For any topological polynomial $f_\sim$, the following claims are
equivalent.
  \begin{enumerate}
  \item \label{condmap} $f_\sim$ is condense.
  \item \label{densecont} The orbit of every continuum $K \subset
    J_\sim$ is dense.
  \item \label{denseint} The orbit of every interval $I \subset
    J_\sim$ is dense.
  \item \label{nope} There are no proper periodic continua in $J_\sim$.
  \end{enumerate}
  Moreover, if these conditions are satisfied, then the set of all
points with condense orbits is residual in every interval $I\subset
J_\sim$.
\end{lemma}

\begin{proof}
  Since every subcontinuum of $J_\sim$ contains an interval, it is
  clear that (\ref{denseint}) and (\ref{densecont}) are equivalent.
  If a point $x \in J_\sim$ has condense orbit and $K \subset J_\sim$
  is a continuum, then $x$ must enter $K$, and the orbit of $K$ is dense.
  This shows that (\ref{condmap}) implies (\ref{densecont}). Moreover, by Lemma~\ref{l:condens1},
  (\ref{condmap}) implies (\ref{nope}).

  Let us show that (\ref{densecont}) and (\ref{nope}) are
equivalent. Suppose that (\ref{densecont}) holds and let $K$ be a
periodic continuum $K$. Then $K$ has to have a dense orbit which by
Lemma~\ref{l:condens1} implies that $K=J_\sim$. Suppose that
(\ref{nope}) holds and let $L\subset J_\sim$ be a continuum. By
\cite{blolev99} there exist $m$ and $n>0$ such that $f^m_\sim(L)\cap
f_\sim^{m+n}(L)\ne \0$. Then the set
$\ol{\bigcup^\infty_{i=0}f^{m+ni}_\sim(L)}=T$ is a periodic
continuum which by the assumption coincides with $J_\sim$. Hence $L$
has a dense orbit as desired.

  Let us show that (\ref{densecont}) implies (\ref{condmap}).  If
  $J_\sim$ has a bounded complementary domain $U$, then we may assume
  that $\bd(U)$ is periodic. By Lemma~\ref{l:condens1} we conclude
  that $\bd(U)=J_\sim$, so $f_\sim$ is conjugate to $z \mapsto z^d$
  and condense. Therefore we may assume that $J_\sim$ is a dendrite.
  Let $\{A_i\,\mid\,i \ge 0\}$ be a countable collection of closed
  arcs such that any continuum $K \subset J_\sim$ contains some $A_s$.
  For convenience, we choose the sequence $\{A_i\}$ so that no element of the sequence
  contains an endpoint of $J_\sim$.

  Let $I \subset J_\sim$ be an arc; we will show for each $s \ge 0$
  that $B_s = \{x \in I \, \mid \, f_\sim^k(x) \in A_s \text{ for some
    $k$}\}$ is an open and dense subset of $I$. Let $\alpha$ denote a
  fixed cutpoint of $J_\sim$.  It follows that, for $i$ sufficiently
  large, $\alpha \in f^i(I)$.  This is because no subcontinuum of
  $J_\sim$ is wandering, i.e., there exists $s, n$ such that
  $f^s_\sim(I)\cap f_\sim^{s+n}(I)\ne\0$ \cite{blolev99}. By
  Lemma~\ref{l:condens1}, for some $M \ge 0$ we have $\alpha \in
  f^{s+Mn}(I)$; since $\alpha$ is fixed, $\alpha \in f^i(I)$ for all
  $i\ge s+Mn$.

  There exist components $K$ of $J_\sim \setminus A_s$ such that every
  arc intersecting $K$ and containing $\alpha$ also contains a
  subinterval of $A_s$. Since every continuum in $J$ has a dense
  orbit, there exists $k \ge 0$ such that $\alpha \in
  f_\sim^k(I)$ and $f_\sim^k(I) \cap K \neq \emptyset$.  Hence,
  $f_\sim^k(I)$ intersects $A_s$ in an open subset.  Since $f_\sim^k$
  is finite-to-one, this implies that an open subset of $I$ maps into
  $A_s$.  Since we can repeat this argument on any subinterval of $I$,
  $B_s$ is a dense open subset of $I$.

  By the Baire Category Theorem, $\bigcap_{s \ge 0} B_s$ is then a
  residual (and hence non-empty) subset of $I$; this is the set of
  points in $I$ which eventually map into each $A_s$, and therefore into
  every subcontinuum of $J_\sim$ as desired.
\end{proof}

Powers of condense maps are condense, too.

\begin{lemma}\label{l:condens3}
If $f_\sim$ is condense and $s \ge 1$, then $f_\sim^s$ is condense.
\end{lemma}

\begin{proof}
  By Lemma~\ref{l:condens2} we need to show that any continuum
  $K\subset J_\sim$ has dense $f_\sim^s$-orbit in $J_\sim$. By
  Lemma~\ref{l:condens0} we only need to consider the case that
  $J_\sim$ is a dendrite.  Let $\alpha \in J_\sim$ be a fixed
  cutpoint. By Lemma~\ref{l:condens1} there exists $i \ge 0$ such that
  $\al\in f^i_\sim(K)$; since $\al$ is fixed we may assume that $i=ks$
  for some integer $k$.  Clearly, $(f_\sim^s)^{k+1}(K) \cap
  (f_\sim^{s})^k(K)\ne \0$, since it contains $\alpha$. By
  Lemma~\ref{l:condens1},
  $\bigcup^\infty_{j=0}f_\sim^{js}(f_\sim^{ks}(K))$ is a connected
  condense subset of $J$, so the $f_\sim^{ks}$-orbit of $K$ is
  condense. Since $K$ was an arbitrary continuum in $J_\sim$,
  $f_\sim^s$ is condense by Lemma~\ref{l:condens2}.
\end{proof}

\begin{theorem}\label{t:condens1}
Let $P$ be a polynomial with connected Julia set. Then the following
claims hold.

\begin{enumerate}

\item Suppose that the finest model $J_\sim$ of $J_P$, given by a
    lamination $\sim$, is non-degenerate, all points of $J_\sim$ are of
    finite order, and $f_\sim$ is condense. Then $J_P$ is locally connected and
    $P|_{J_P}$ is conjugate to $f_\sim$.

\item Suppose that $P|_{J_P}$ is condense. Then $P$ has no proper periodic subcontinua (in particular, $P$ is
    non-renormalizable), $J_P$ is locally connected and $P$ is conjugate to $g_P$ from Theorem~\ref{t:model}.

\end{enumerate}

\end{theorem}

Observe, that by this theorem $P|_{J_P}$ satisfies
Lemmas~\ref{l:condens0} - \ref{l:condens3}. Observe also, that by
Theorem~\ref{kiwistruct} (1) holds for polynomials associated with
condense critical portraits having aperiodic kneading.

\begin{proof}
  (1) Let $m:J_P\to J_\sim$ be the finest monotone map to a locally
connected continuum defined in Theorem~\ref{t:model}.
   Since the order of any periodic point $p \in J_\sim$ is finite,
  by \cite[Lemma 37]{bco11} the set $m^{-1}(p)$ is a repelling or parabolic
  periodic point. Hence, $P$ has no Cremer points: if $U$ were a
  periodic Siegel domain of $P$, then $m(\bd(U))$ would be a
  periodic subcontinuum of $J_\sim$ homeomorphic to a circle on
  which the appropriate power of the map is an irrational rotation,
  and hence a proper subcontinuum.

  Now we show that $P$ is non-renormalizable.  Indeed, if $P$ is
  renormalizable, then there exists a polynomial-like connected Julia
  set $J'\subsetneqq J_P$ which is a periodic continuum. If $m(J')$ is a point, then it is periodic and
  again by \cite[Lemma 37]{bco11} the set $m^{-1}(m(J'))$ is a point, a contradiction. Hence
  $m(J')$ is a periodic continuum in $J_\sim$. Clearly, $m(J')\ne J_\sim$. This
  contradicts Lemma~\ref{l:condens1} and
  Lemma~\ref{l:condens2} and shows that $P$ is non-renormalizable. Hence
  $J_P$ is locally connected \cite{kvs09}.  By
  Theorem~\ref{t:model},  $P|_{J_P}$ and $f_\sim$ are
  conjugate as required.

  (2) Assume now that $P|_{J_P}$ is condense.  Let us show that $J_P$
  has no proper periodic subcontinua. Indeed, let $A\subset J_P$ be a
  periodic continuum. Then the (finite) union $B$ of its images must
  coincide with $J_P$ (because $P|_{J_P}$ is condense). As at least
  one of these images must have non-empty interior, $A$ must
  coincide with $J_P$.

  This fact has several consequences.  To begin with, let us show that
  $J_P$ cannot have Cremer points. Indeed, suppose that $z_0\in J_P$ is
  a periodic Cremer point of period $p$. Then, for any small neighborhood
  $U$ of $z_0$, the component of the set $\{z\,\mid\,P^{kp}(z) \in \overline{U} \text{
  for all $k$}\}$ containing $z_0$, called a \emph{hedgehog}, is a
  proper periodic subcontinuum of $J_P$ \cite{pm97}, contradicting that $P|_{J_P}$ has no
  proper periodic subcontinuum.

  Now let us show that $P$ cannot have Siegel domains either. Since
$J_P$ contains no proper periodic subcontinua, then any periodic
Siegel domain $U$ of $P$ must be such that $\bd(U)=J_P$. By J.
Rogers' result \cite{rog92}, there are two cases. In the first case,
$P|_{\bd(U)}$ is monotonically semiconjugate to an irrational
rotation which contradicts the fact that $\bd(U)=J_P$. In the second case,
$\bd(U)$ is an \emph{indecomposable} continuum (i.e.,
cannot be represented as $A\cup B$ where $A$ and $B$ are proper
subcontinua of $\bd(U)$). Then, given a point $x\in \bd(U)$, one can
define the \emph{composant} of $x$ in $\bd(U)$, that is the union of
all proper subcontinua of $\bd(U)$ containing $x$. Then it is known
\cite[Theorem 11.15]{nadl92} that distinct non-degenerate composants
of $\bd(U)$ are pairwise disjoint and there are uncountably many of them.
Since the orbit of $x$ can only enter countably many composants of
$\bd(U)$, we have a contradiction with the assumption that $P|_{J_P}$ is condense.
Hence, $P$ does not have Siegel domains.

\begin{comment}

Also, $J_P$ has no attracting or parabolic Fatou domains $U$ because
otherwise we may assume that $U$ is periodic and then by Claim 1
$\bd(U)=J_P$ while by \cite{ry08} $\bd(U)$ is a simple closed curve.

In this, $P$ is conjugate to $z
  \mapsto z^d$ \cite[Theorem 9.4.3]{beardon}, so $J_P$ is locally connected and we are
  done.

  First, let us show that either $J_P$ is non-separating, or $J_P$ is a Jordan curve and $P_{J_P}$
  is conjugate to the map $z^d$ restricted to the unit circle.
  Indeed, if $U$ is any periodic bounded Fatou
  domain, then $J_P = \partial U$ by the previous paragraph.  If $J_P$
  is the boundary of more than one bounded Fatou domain, then $J_P$ is
  a \emph{Lakes of Wada continuum}, and therefore contains an
  indecomposable subcontinuum \cite[Theorem X.V.11]{kurat68}.
  However, indecomposable continua contain uncountably many disjoint
  subcontinua \cite[Theorem 11.15]{nadl92}, contradicting the
  condensity of $P|_{J_P}$.  Hence, either $J_P$ is non-separating, or $K_P$
  contains exactly one invariant bounded Fatou domain $U$ such that $\partial U=J_P$.
\end{comment}

  Since $J_P$ has no proper periodic subcontinua, $P$ is non-renormalizable.
  Thus, as before, all this implies that $J_P$ is locally connected \cite{kvs09}.
  The rest follows from Theorem~\ref{t:model}.
\end{proof}

\section{Family of critical WT-portraits}\label{famwt}

First we show that condense laminations are residual in $\A_3$.

\begin{theorem}\label{t:condenscript}
  Let $\Ta\in \A_d$ be a critical portrait which consists of $d-1$
  critical chords whose orbits are dense in $\uc$. Then $\Ta$ has
  aperiodic kneading, is condense, and any polynomial $P$ associated
  to $\Ta$ has locally connected Julia set $J_P$ so that $P|_{J_P}$ is
  conjugate to $f_{\sim_\Ta}|_{J_{\sim_\Ta}}$.
\end{theorem}

\begin{proof}
  Let us show that $\Ta$ has aperiodic kneading.  Indeed, the orbit of
  any critical leaf $\ell$ comes arbitrarily close to the fixed point
  $0$. Hence, if $C$ is the $\Ta$-unlinked class of $0$, then the
  itinerary of $\ell$ includes arbitrarily long segments consisting of
  $C$. This implies that $\Ta$ has aperiodic kneading, and
  Theorem~\ref{kiwistruct} applies.  Let $\sim$ denote the lamination
  generated by $\Ta$.

  Let us show that $\Ta$ is condense.  Take an arc $I\subset J_\sim$
  and consider its orbit. By \cite{blolev99} there are positive
  numbers $m, k$ with $f_\sim^m(I)\cap f_\sim^{m+k}(I)\ne 0$. Consider
  the connected set
  $A_0=\bigcup^\infty_{i=0}f_\sim^{m+ki}(I)$. Clearly,
  $\ol{A_0}=B\subset J_\sim$ is a subdendrite of $J_\sim$ and
  $f_\sim^k(B)\subset B$. Let us show that $f_\sim^k|_B$ has a
  critical point $c$. Indeed, by Theorem~7.2.6 of \cite{bfmot11} there
  are infinitely many periodic cutpoints of $f_\sim^k|_B$; let $Q
  \subset B$ be an arc joining some pair $x$ and $y$ of such periodic
  cutpoints.  If $f_\sim^k|_B$ has no critical points, then some power
  of $f_\sim^k|_Q$ is a homeomorphism and there must exist a point
  $z\in Q$ attracting for $g$ from at least one side, which is
  impossible.  Hence, $B$ contains a critical point of $f_\sim^k$.  By
  the assumptions on $\Ta$, $B$ contains a point with dense orbit, so
  $I$ has a dense orbit.  Since $I$ was arbitrary, we conclude by
  Lemma~\ref{l:condens2} that $\Ta$ is condense.

  Since $\Ta$ satisfies Theorem~\ref{kiwistruct}, and since $\sim =
  \sim_P$ by Remark~\ref{rem:kiwimodel},
  it follows from Theorem~\ref{t:condens1} (1) that $J_\sim$ is
  locally connected and that $P|_{J_P}$ is conjugate to $f_\sim$.

%It follows from \cite[Lemma 18.8]{milnor} that
%  $B$ contains a cutpoint $c$ which is critical under $f^k$ (or else
%  $f^k|_{B}$ is a forward-expanding homeomorphism).

% Let $p:\uc\to J_\sim$ denote the quotient map corresponding to $\sim$. Choose $\al\in\uc$ with
%   $p(\alpha)=c$; it follows from our assumptions on $\Ta$ that the
%   orbit of $\al$ is dense in $\uc$. Therefore,
%   the orbit of $c$ is dense in $J_\sim$, as is the orbit of $B$ (which
%   contains $c$). This implies that
%   $I$ has a dense orbit in $J_\sim$. Since $I$ is an arbitrary arc in
%   $J_\sim$, we conclude by Lemma~\ref{l:condens2} that $\Ta$ is
%   condense. Moreover, by the claim established in the paragraph right
%   after Theorem~\ref{t:model} we have that $\sim_P=\sim_\Ta$. Now this
%   and Theorem~\ref{kiwistruct} show that Theorem~\ref{t:condens1}(1)
%   applies here and allows us to conclude that $J_P$ is locally
%   connected and $P|_{J_P}$ is condense as desired.
\begin{comment}
Let $p:\uc\to J_\sim$
  denote the quotient map. If $p(\alpha)=c$  for some $\al\in\uc$, then the sequence
   $(\sigma_{d}^{kn}(\alpha))_{n=1}^\infty$
  is dense in $\ucirc$ (cf. Lemma~\ref{l:condens3}).  Therefore,
  $(f_\sim^{kn}(c))_{n=1}^\infty$ is dense in $J_\sim$ and we conclude
  that $B = J_\sim$.  By Lemma~\ref{l:condens2} it follows that
  $f_\sim$ is condense, so by Theorem~\ref{t:condens1} $J_P$ is
  locally connected and $P|_{J_P}$ is condense.
\end{comment}
\end{proof}

Since the set of critical portraits consisting of $d-1$ critical leaves
with dense orbits in $\uc$ is residual in $\A_d$, we obtain the
following corollary.

\begin{corollary}
  A residual subset of critical portraits in $\A_d$ correspond to
  polynomials whose restrictions to their Julia sets are condense.
\end{corollary}

Recall that a lamination with wandering $k$-gons ($k\ge 3$) is
called a \emph{WT-lamination}. A critical portrait, compatible with
a WT-lamination, is called a \emph{WT-critical portrait}; $\WT_3$ is
the set of all cubic WT-critical portraits. By Theorem~\ref{main1},
$\WT_3$ is a dense and locally uncountable subset of $\A_3$.

Now we show that $\WT_3$ is a meager subset of $\A_3$.  We will do
so by showing that the set of critical portraits in $\WT_3$
compatible with a wandering triangle of area at least $\frac 1 n$ is
disjoint from a particular dense subset of critical portraits.  The
dense subset we consider, called $\K$, is the set of critical
portraits consisting of two leaves $\{\mc, \md\} \in \A_3$ such that
the orbits of $\mc$ and $\md$ are dense, neither $\mc$ nor $\md$
maps to an endpoint of the other, and $\mc$ and $\md$ eventually map
to the same point.

\begin{lemma}\label{lem:classK}
The set $\K$ is dense in $\A_3$. All orbit portraits $\Ta \in \K$
have aperiodic kneading. The critical classes of the lamination
$\sim_\Ta$ generated by $\Ta$ are leaves.
\end{lemma}
\begin{proof}
  The fact that $\K$ is dense in $\A_3$ is easy and left to the
  reader. Consider some $\Ta = \{\mc, \md\} \in \K$. By
  Theorem~\ref{t:condenscript}, $\Ta$ has aperiodic kneading.  Let
  $\g$ be the critical $\sim_\Ta$-class containing $\mc$, and $\h$ the
  critical class containing $\md$. It is easy to see that if $\g$
  contains at least three points, then $|\si(\g)|\ge 2$. Indeed, consider two cases.
  If $\g$ maps to its image
  in the two-to-one fashion, then $|\si(\g)|\ge 2$ is obvious. If $\g$ maps to its image in the three-to-one fashion then
  $\g=\h$ contains four endpoints of the leaves $\mc$ and $\md$, so again $|\si(\mg)| \ge 2$.
  Similarly, if $|\h|\ge 3$ then $|\si(\h)|\ge 2$.

  Suppose for contradiction that $\g$ contains at least three points.
  We will first show that then all forward images of all critical classes
  of $\sim_\Ta$ are non-degenerate.  Indeed, note that neither $\g$
  nor $\h$ may eventually map onto itself, since the orbits of $\mc$
  and $\md$ are dense in $\ucirc$.  This further implies that, if
  $\g$ maps onto $\h$, then $\h$ cannot map onto $\g$.  We
  consider three cases.
  \begin{enumerate}
  \item Suppose that $\g=\h$. Since $|\si(\g)| \ge 2$ and $\mg$ is not periodic,
    it is not precritical, so $|\si^l(\g)|=|\si^l(\h)|\ge 2$ for all $l\ge 0$.
  \item Suppose that $\si^k(\g)=\h$ for some $k\ge 1$. Since $|\si(\g)|\ge 2$ and
    $\mc$ never maps into $\md$, we see that
    $\h$ contains at least three points (the endpoints of $\md$ and
    the point $\si^k(\mc)$). Therefore by the above $|\si(\h)| \ge 2$.  As noted before,
    $\h$ is not precritical, so $|\si^k(\mg)|$ and $|\si^k(\mh)|$
    are both at least two for all $k$.
  \item If $\mg$ never maps onto $\mh$, then $|\si^k(\mg)| \ge 2$ for
    all $k$, since $\mg$ is not precritical and contains at least
    three points.  Since $\mc$ and $\md$ have a common image, so do
    $\mg$ and $\mh$, and $|\si^k(\mh)| \ge 2$ for all $k$.
  \end{enumerate}

We will use the metric where the distance between two
  points on $\uc$ is the length of the shortest arc in $\uc$ joining
  them. By the diameter of a chord we will mean the distance between
  its endpoints. Let us show that $\diam(\si^k(\g))$ is bounded away
  from $0$. It is easy to see that, for any chord $\ell'$,
  \begin{equation}\label{stupidsasha}
  \diam(\si(\ell'))=\begin{cases}
  3\diam(\ell') & \text{if $\diam(\ell')\le 1/6$}\\
  3|\diam(\ell')-1/3| & \text{if $1/6\le \diam(\ell')$.}
  \end{cases}
  \end{equation}
  This implies that $\diam(\si(\ell')) \ge \diam(\ell')$ if and only
  if $\diam(\ell') \le 1/4$. Hence, every class of diameter less
  than $1/4$ maps to a class of larger diameter. Let $\ell$ be the
  chord on $\bd(\ch(\mg))\cup\bd(\ch(\mh))$ of length closest to
  $1/3$; since $\si(\mg)$ and $\si(\mh)$ are non-degenerate, $\e =
  |\diam(\ell) - 1/3|$ is positive.  Since $\sim$-classes are
  unlinked, $|\diam(\ell')-1/3|\ge \e$ for any chord $\ell'$ from
  the boundary of the convex hull of a $\sim$-class.  Hence, by
  Equation~\ref{stupidsasha} no class of diameter at least $1/4$ has
  an image of diameter less than $3 \e$.  In particular,
  $\diam(\si^k(\g)) \ge 3 \e$ for all $k$.

  Since the convex hulls of classes are dense in $\disk$, we can choose
  a class $\mk$ so that there exists a component $A$ of $\uc\sm \mk$
  of diameter less than $\e$.  Since $\diam(\si^k(\g)) \ge 3 \e$,
  the orbit of $\mg$ can never enter $A$.  This contradicts that the
  orbit of $\mc$ is dense.  We conclude that the classes $\mg$ and
  $\mh$ are leaves.
\end{proof}

\begin{theorem}\label{t:meager} The set $\WT_3$ is of first category in $\A_3$.
\end{theorem}

\begin{proof}
  Let $\W_n$ be the set of critical portraits $\Ta\in \WT_3$ such
  that there is a $\sim_\Ta$-class $\T$ which is a wandering
  triangle and $\ch(\T)$ has area at least $1/n$. We will show that
  $\W_n$ is nowhere dense by showing that $\overline{\W_n} \cap \K =
  \emptyset$.

  By Theorem~\ref{bl}, $\W_n$ is disjoint from $\K$ for every $n$.
  Suppose that there is a sequence $(\Ta_i)_{i=1}^\infty$ of
  elements of $\W_n$ which converges to a critical portrait $\Ta =
  \{\mc, \md\}\in \K$. By way of contradiction we will prove that
  $\Ta\in \K$ is impossible. For each $i$ set $\sim_i=\sim_{\Ta_i}$
  and let $\T_i$ be a wandering triangle in $\sim_i$ such that
  $\ch(\T_i)$ has area at least $1/n$. We may assume that
  $(\T_i)_{i=1}^\infty$ converges to a triangle $\T=\{a, b, c\}$,
  with area of $\ch(\T)$ at least $1/n$.

  Let us prove that $\T$ is contained in some $\sim_\Ta$-class $\T'$.
  For any fixed $k$ and for every $i$, $\si^k(\T_i)$ is contained in a
  $\sim_i$-unlinked class.  Hence, we have the following.

  \smallskip

  \noindent{\textbf{Claim A.}} \emph{For every $k\ge 0$ we have that $\lim_{i \to \infty}
  \si^k(\T_i) = \si^k(\T)$ is a subset of the \emph{closure} of a
  $\sim_\Ta$-unlinked class.}

  \smallskip

  Let us show that every pair of vertices of $\T$ are
  $\sim_\Ta$-equivalent. The crucial observation here is that the
  orbits of the vertices of $\T$ each cannot intersect $\bigcup \Ta$
  more than once ($\bigcup \Ta$ denotes the collection of the
  endpoints of critical leaves in $\Ta$). Indeed, if an orbit
  intersects $\bigcup \Ta$ twice, then an endpoint of one critical
  leaf eventually maps to an endpoint of either the same critical
  leaf or the other critical leaf, contradicting that $\Ta \in \K$.

  By Lemma~\ref{lem:classK}, $\mc$ and $\md$ are distinct classes of
  $\sim_\Ta$. Let us show that there exists $m\ge 1$ such that
  $\si^m_3(\mc)$ and $\si^m_3(\md)$ belong to distinct $\Ta$-unlinked
  classes. Indeed, otherwise $\si_3(\mc)$ and $\si_3(\md)$ are
  distinct points of $\uc$ (they are distinct because $\mc$ and $\md$
  are disjoint) which belong to the same $\sim_\Ta$-class; this
  contradicts the conclusion of Lemma~\ref{lem:classK} that the
  $\sim_\Ta$-class of $\si_3(\mc)$ is a singleton.

  We can now show that $a \sim_\Ta b$ for any vertices $a, b \in
  \T$. Indeed, suppose first that for some $l \ge 0$ the points $\si^l(a)$ and
  $\si^l(b)$ are contained in the same critical class. By the
  previous observation that each orbit of a vertex of $\T$ can
  intersect $\bigcup \Ta$ at most once, for every $k \neq l$ we have
  that $\si^k(a)$ and $\si^k(b)$ are not in $\bigcup \Ta$. We can
  therefore, by Claim A, choose two opposite (positive and negative) sides of $a$
  and $b$ so that the corresponding one-sided itineraries of $a$ and
  $b$ coincide, so $a \sim_\Ta b$.

  On the other hand, suppose $a$ and $b$ never simultaneously map to
  the endpoints of a critical leaf. We now note that, for any $l \ge
  0$, the points $\si^l(a)$ and $\si^l(b)$ cannot belong to
  different critical classes, or else $\si^{l+m}_3(a)$ and
  $\si^{l+m}_3(b)$ would belong to distinct $\Ta$-unlinked classes,
  a contradiction with Claim A. This again implies that we can chose
  two distinct sides (positive and negative) such that the
  corresponding one-sided itineraries of $a$ and $b$ coincide. Thus,
  in any case $a \sim_\Ta b$ and so $\T$ is contained in a
  $\sim_\Ta$-class $\T'$.

  By Theorem~\ref{kiwistruct}, $\T'$ is finite. Since $\Ta \in \K$,
  $\T'$ is not wandering by Theorem~\ref{bl}, and $\T'$ is not
  precritical by Lemma~\ref{lem:classK}. Hence, $\T'$ is
  preperiodic, and either $\T$ itself is preperiodic or its future
  images cross each other inside $\disk$. As the latter is
  impossible by continuity, we may assume that there exist powers
  $s$ and $t>0$ such that $\si^s_3(\T) = \si^{s+t}_3(\T)$. Again by
  continuity $\si^s_3(\T_i)$ and $\si_3^{s+t}(\T_i)$ approach
  $\si_3^s(\T)$ in the Hausdorff metric while the area of $\ch(\T)$
  is at least $1/n$. For geometric reasons this contradicts that
  $\si_3^t(\T_i)$ and $\si_3^{s+t}(\T_i)$ are disjoint for all $i$.
  Therefore, $\Ta \notin \K$.

  We have established that $\W_n$ is nowhere dense in
  $\A_3$, so $\bigcup_{n=1}^\infty \W_n = \WT_3$ is a first category
  subset of $\A_3$.
\end{proof}

\medskip
% The data information below will be filled by AIMS editorial staff
Received xxxx 20xx; revised xxxx 20xx.
\medskip

\end{document}